\RequirePackage{amsmath}
\documentclass[a4paper]{article}
\usepackage{graphicx}

\usepackage{amssymb}
\usepackage[english]{babel}
\usepackage{amsthm}
\usepackage{fourier}
\usepackage{graphicx}
\usepackage{color}
\usepackage{tikz}
\usepackage{tikz-cd}
\usetikzlibrary{patterns}

\DeclareMathOperator{\depth}{depth}
\DeclareMathOperator{\conv}{conv}
\DeclareMathOperator{\relint}{int}
\DeclareMathOperator{\wgt}{weight}

\newtheorem {thm}{Theorem}
\newtheorem*{thm*}{Theorem}
\newtheorem*{conj*}{Conjecture}
\newtheorem {lem}[thm]{Lemma}
\newtheorem {prop}[thm]{Proposition}
\newtheorem {cor}[thm]{Corollary}

\begin{document}

\title{An improvement on the Rado bound for the centerline depth%
\thanks{The first author was supported by ERC Advanced Research Grant no. 267165 (DISCONV).}}

\author{Alexander Magazinov\thanks{Tel Aviv University, Faculty of Exact Sciences, School of Mathematics.\newline%
E-mail: {\it magazinov@post.tau.ac.il}} \and %
Attila P\'or\thanks{West Kentucky University, Bowling Green, Department of Mathematics.\newline%
E-mail: {\it attila.por@wku.edu}}%
}

\maketitle

\begin{abstract}
Let $\mu$ be a Borel probability measure in $\mathbb R^d$. For a $k$-flat $\alpha$ consider the value $\inf \mu(H)$, where $H$ runs through
all half-spaces containing $\alpha$. This infimum is called the {\it half-space depth} of $\alpha$.

Bukh, Matou\v{s}ek and Nivasch conjectured that for every $\mu$ and every $0 \leq k < d$ there exists a $k$-flat with
the depth at least $\tfrac{k + 1}{k + d + 1}$. The Rado Centerpoint Theorem implies
a lower bound of $\tfrac{1}{d + 1 - k}$ {\it (the Rado bound)}, which is, in general, much weaker.
Whenever the Rado bound coincides with the bound conjectured by Bukh, Matou\v{s}ek and Nivasch, i.e., for $k = 0$ and $k = d - 1$, it
is known to be optimal.

In this paper we show that for all other pairs $(d, k)$ one can improve on the Rado bound.
If $k = 1$ and $d \geq 3$ we show that there is a 1-dimensional line with the depth at least $\tfrac{1}{d} + \tfrac{1}{3d^3}$.
As a corollary, for all $(d, k)$ satisfying $0 < k < d - 1$ there exists a $k$-flat with depth at least
$\tfrac{1}{d + 1 - k} + \tfrac{1}{3(d + 1 - k)^3}$.
\end{abstract}



\section{Introduction}

Let $\alpha$ be a $k$-flat and $\mu$ a Borel probability measure in $\mathbb R^d$ ($0 \leq k < d$). Define the {\it depth} of $\alpha$ as follows:
$$\depth_{\mu}(\alpha) = \inf \{\mu(H) : \text{$H$ is a closed half-space, $\alpha \subset H$} \}.$$
Sometimes the depth defined above is called {\it half-space depth} or {\it Tukey depth} in order to distinguish it from other commonly used notions of depth.
We will write simply $\depth(\alpha)$ if the measure is clear from the context.


One of the most important results concerning the notion of half-space depth is the Rado Centerpoint Theorem.

\begin{thm*}[Rado Centerpoint Theorem,~\cite{rad}]
For every Borel probability measure $\mu$ in $\mathbb R^d$ there exists a point $x$ such that $\depth(x) \geq \tfrac{1}{d + 1}$.
\end{thm*}

Bukh, Matou\v{s}ek and Nivasch proposed the following conjecture, which, if proved, would be a generalization of the Rado Theorem.

\begin{conj*}[Bukh, Matou\v{s}ek, Nivasch,~\cite{bmn}]
Let $(d, k)$ be a pair of integers with $0 \leq k < d$. Then for every Borel probability measure $\mu$ in
$\mathbb R^d$ there exists a $k$-flat $\alpha$ in $\mathbb R^d$ (a {\it centerflat}) such that
\begin{equation}\label{eq:bmn-conj}
\depth(\alpha) \geq \frac{k + 1}{k + d + 1}.
\end{equation}
\end{conj*}

The conjecture is true for $k = 0$ (in this case the conjecture turns exactly into the Rado Theorem),
$k = d - 1$ (a trivial case), and $k = d - 2$ (a case settled by Bukh, Matou\v{s}ek and Nivasch~\cite{bmn} themselves).

A result by Klartag~\cite{kla} implies that, if $d - k$ is fixed, then for every $\varepsilon > 0$, with $d$ sufficiently large
depending on $\varepsilon$, and for every Borel probability measure $\mu$ in $\mathbb R^d$ there exists a $k$-flat $\alpha$ in $\mathbb R^d$
such that
$$\depth(\alpha) > \frac{1}{2} - \varepsilon.$$

One can see that for $k = 0$ and $k = d - 1$ the constant $\tfrac{k + 1}{k + d + 1}$ in \eqref{eq:bmn-conj} cannot be increased.
This is also the case for $k = 1$, as shown by Bukh and Nivasch~\cite{bn}.

For the purposes of our paper it will be convenient to think about a depth of a flat in terms of projections. If $\mu$ is a Borel probability measure in
$\mathbb R^d$, and $\alpha$ is a $k$-flat, we will write $\pi_{\alpha}$ for the orthogonal projection from $\mathbb R^d$ onto the $(d - k)$-space
$\beta = \alpha^{\bot}$ (i.e., $\pi_{\alpha}(\alpha)$ is a single point). Let $\mu_{\alpha}$ be a projection of $\mu$ along $\alpha$, i.e.,
a measure in $\beta$ such that for every Borel set $X \subseteq \beta$ one has
$$\mu_{\alpha}(X) = \mu(\pi_{\alpha}^{-1}(X)).$$
(Of course, $\mu_{\alpha}$ is a Borel probability measure in $\beta$.) Then for the point $o = \pi_{\alpha}(\alpha)$ one has the identity
$$\depth_{\mu_{\alpha}}(o) = \depth_{\mu}(\alpha).$$
We also note that a projection of a measure is sometimes called a {\it marginal}, see~\cite{kla}.

The Rado Centerpoint Theorem implies that for every $d$, $k$ and $\mu$ as above one can find a $k$-flat $\alpha$ such that
\begin{equation}\label{eq:triv-bound}
\depth(\alpha) \geq \frac{1}{d - k + 1}.
\end{equation}
(In fact, such a $k$-flat exists in any $k$-dimensional direction.) The bound of~\eqref{eq:triv-bound} will be called the
{\it Rado bound}.

In this paper we prove that for $k = 1$ the Rado bound~\eqref{eq:triv-bound} is not optimal, except for the case $d = 2$.
Namely, we have the following result:

\begin{thm}\label{thm:main}
For every $d \geq 3$ and for every Borel probability measure $\mu$ in $\mathbb R^d$ there exists a (1-dimensional) line $\ell$ with
$$\depth(\ell) \geq \frac{1}{d} + \frac{1}{3d^3}.$$
\end{thm}

Theorem~\ref{thm:main} also implies that the Rado bound~\eqref{eq:triv-bound} is optimal {\it only} for the cases $k = 0$ and
$k = d - 1$, as stated in the following Corollary~\ref{cor:main}. We emphasize that there is still a huge gap between the
inequality~\eqref{eq:main-cor} we were able to prove, and the conjectured inequality~\eqref{eq:bmn-conj}.

\begin{cor}\label{cor:main}
For every $d \geq 3$, every $k$ such that $1 \leq k \leq d - 2$ and every Borel probability measure $\mu$ in $\mathbb R^d$ there exists a
$k$-dimensional flat $\alpha$ with
\begin{equation}\label{eq:main-cor}
\depth(\alpha) \geq \frac{1}{d - k + 1} + \frac{1}{3(d - k + 1)^3}.
\end{equation}
\end{cor}

\begin{proof}[Reduction to Theorem~\ref{thm:main}]
Choose an arbitrary $(k - 1)$-dimensional flat $\beta$. After projecting along $\beta$ onto $\mathbb R^{d - k + 1}$ we can apply
Theorem~\ref{thm:main}. Namely, we conclude that there is a line $\ell \subset \mathbb R^{d - k + 1}$ such that
$$\depth_{\mu_{\beta}}(\ell) \geq \frac{1}{d - k + 1} + \frac{1}{3(d - k + 1)^3}.$$
To finish the proof it is enough to put
$$\alpha = \pi_{\beta}^{-1}(\ell).$$

In the rest of the paper we prove Theorem~\ref{thm:main}. The body of the argument is contained in Sections 2--4. Sections
5--9 incorporate the proofs of the technical statements declared in Section 2.

\end{proof}

\section{Geometric part: statements}\label{sect:outline}

First, it will be convenient for us to prove Theorem~\ref{thm:main} for the $(d + 1)$-dimensional space rather than for the $d$-dimensional.
Next, we aim for a proof by contradiction. Therefore we assume that for every one-dimensional direction $\ell$ no point
of the $d$-dimensional plane $\ell^{\bot}$ has depth (with respect to the projected measure $\mu_{\ell}$)
$\tfrac{1}{d + 1} + \tfrac{1}{3(d + 1)^3}$ or greater. For brevity, we will write
$$a_0 = a_0(d + 1) = \frac{1}{d + 1} + \frac{1}{3(d + 1)^3}.$$

\subsection{Nice measures}

A Borel probability measure $\mu$ in a Euclidean $d$-space $V$ will be called {\it a nice measure} if it has a density function $f_{\mu}: V \to \mathbb R$
satisfying the following properties:
\begin{enumerate}
  \item $f_{\mu}$ is continuous.
  \item $f_{\mu}(x) > 0$ for every $x \in V$.
  \item There exist $C_1, C_2 > 0$ such that $f_{\mu}(x) < C_1e^{-C_2|x|}$ for every $x \in V$.
\end{enumerate}
We supply the space $\mathcal M(V)$ of nice measures in $V$ with a metric --- the $L^1$ distance between density functions:
$$\| \mu - \mu' \| = \| f_{\mu} - f_{\mu'} \|_{L^1} = \int\limits_{V} |f_{\mu}(x) - f_{\mu'}(x)|\, dx.$$

Let $Iso(V)$ be the group of all isometries of $V$. Then every element $F \in Iso(V)$ admits a natural push-forward $F_* : \mathcal M(V) \to \mathcal M(V)$.
Namely, we define the measure $F_*(\mu)$ via
$$F_*(\mu)(X) = \mu(F^{-1}(X) \quad \text{for every Borel set $X \subseteq V$.}$$
Recall that $Iso(V)$ has a natural topology. Every $F \in Iso(V)$ can be represented as $F(x) = Ax + v$, where $A \in O(V)$,
$v \in V$, and the convergence $F \to Id$ is equivalent to the simultaneous convergence $A \to Id$ and $v \to \mathbf 0$.

In the next proposition (Proposition~\ref{prop:nice_mes}) we collect the most important facts about nice measures that we will use in the paper.
We omit the proof, as it is plainly standard.

\begin{prop}\label{prop:nice_mes}
The following assertions hold:
\begin{enumerate}
  \item Let $\mu \in \mathcal M(\mathbb R^d)$. Then
        $$\lim\limits_{F \to Id} \| \mu - F_*(\mu) \| = 0,$$
        where $F$ runs through $Iso(\mathbb R^d)$.
  \item Let $\alpha \subset \mathbb R^d$ be a $k$-flat, where $k < d$, $\mu \in \mathcal M(\mathbb R^d)$. Then $\mu_{\alpha}$ is a nice measure.
  \item Let $\alpha \subset \mathbb R^d$ be a fixed $k$-flat. Consider $\mu_{\alpha} : \mathcal M(\mathbb R^d) \to \mathcal M(\alpha^{\bot})$
        as a function of $\mu$. Then $\mu_{\alpha}$ is continuous in $\mathcal M(\alpha^{\bot})$.
  \item Consider $\depth_{\mu}(x): \mathcal M(\mathbb R^d) \times \mathbb R^d \to \mathbb R$ as a function of $\mu$ and $x$. Then $\depth_{\mu}(x)$ is
        continuous in $\mathcal M(\mathbb R^d) \times \mathbb R^d$ (with the product topology).
\end{enumerate}
\end{prop}

In the proof of Theorem~\ref{thm:main} we will assume that $\mu$ is nice. Then the case of an arbitrary $\mu$ will follow from a
standard approximation argument.

\subsection{Properties of Tukey medians}

Write
$$\mathcal M_a(V) = \{ \nu \in \mathcal M(V): \sup\limits_{x \in V} \depth_{\nu}(x) < a \}.$$
We will consider $\mathcal M_a(V)$ as a subspace of $\mathcal M(V)$ with the induced topology.
Note that the Rado Centerpoint Theorem implies $\mathcal M_a(V) = \varnothing$ for all $a \leq \tfrac{1}{d + 1}$.

Recall the notation $a_0 = \tfrac{1}{d + 1} + \tfrac{1}{3(d + 1)^3}$. In order to prove Theorem~\ref{thm:main} in $\mathbb R^{d + 1}$
by contradiction we have to assume that
$$\mu_{\ell} \in \mathcal M_{a_0}(\ell^{\bot})$$
for every line $\ell \in \mathbb RP^d$.

For $\nu \in \mathcal M(V)$ we will call a point $o \in V$ a {\it Tukey median} of $\nu$ if $\depth_{\nu}(o) = \sup\limits_{x \in V} \depth_{\nu}(x)$.
The following Lemma~\ref{lem:median} concerns the properties of Tukey medians. The idea of such statement is certainly not new,
see, for instance,~\cite{BF1984}.

\begin{lem}\label{lem:median}
Let $V$ be a Euclidean $d$-space, $a \in \left( \tfrac{1}{d + 1}, \tfrac{1}{d} \right)$. Then the following assertions hold.
\begin{enumerate}
	\item Let $\nu \in \mathcal M_a(V)$, $o$ be a Tukey median of $\nu$. Then there exists $d + 1$ half-spaces $H_1, H_2, \ldots, H_{d + 1} \subset V$
          such that $o \in \partial H_i$, $\bigcap\limits_{i = 1}^{d + 1} H_i = o$, and $\nu(H_i) = \depth_{\nu}(o)$.
	\item For every $\nu \in \mathcal M_a(V)$ the Tukey median of $\nu$ is unique.
    \item Let $o(\nu)$ denote the Tukey median of $\nu$ for every $\nu \in \mathcal M_a(V)$. Then $\depth_{\nu}(o(\nu))$ depends continuously on $\nu$.
    \item $o(\nu)$ is a continuous function of $\nu$ if $\nu$ runs through $\mathcal M_a(V)$.
\end{enumerate}
\end{lem}

The proof will be given in Section~\ref{sect:centerpoint}.

Define
$$\mathcal M_a^{\circ}(V) = \{ \nu \in \mathcal M_a(V): o(\nu) = \mathbf 0 \}.$$
I.e., $\mathcal M_a^{\circ}(V)$ contains all those measures in $\mathcal M_a(V)$ whose Tukey median is the origin. By Lemma~\ref{lem:median},
assertion 2, for every $\nu \in \mathcal M_a(V)$ there exists a unique translation $F$ such that the translated measure $\nu^{\circ} = F_*(\nu)$
belongs to $\mathcal M_a^{\circ}(V)$ (namely, $F$ is the translation by $-o(\nu)$).

Lemma~\ref{lem:median}, assertion 4, and Proposition~\ref{prop:nice_mes}, assertions 1 and 4, imply that $\nu^{\circ}$ is a continuous
function of $\nu$.

\subsection{Geometry of measures in $\mathcal M_{a_0}^{\circ}(V)$}

Let $V$ be a Euclidean $d$-space. Denote by $\mathcal T(V)$ the set of all unordered $(d + 1)$-tuples $\{ e_1, e_2, \ldots, e_{d + 1} \}$,
$e_i \in V$ such that
$$\dim \conv \{ e_1, e_2, \ldots, e_{d + 1} \} = d; \quad \mathbf 0 \in \relint \conv \{ e_1, e_2, \ldots, e_{d + 1} \}.$$
Of course, $\mathcal T(V)$ can be considered as a topological space with the topology induced from $V^{d + 1} / \mathfrak S_{d + 1}$, where
$\mathfrak S_{d + 1}$ is the symmetric group with the usual action on the $(d + 1)$-th power of $V$.

The main geometric statement of the paper is provided below in Lemma~\ref{lem:structure}. We say that an isometry $F: V_1 \to V_2$ between
two a Euclidean $d$-spaces is {\it linear} if it maps the origin of $V_1$ to the origin of $V_2$. Every such isometry naturally defines
a push-forward map $F_* : M_a^{\circ}(V_1) \to \mathcal M_a^{\circ}(V_2)$ and the map $F: \mathcal T(V_1) \to \mathcal T(V_2)$ resembling the
usual notation:
$$F(\{ e_1, e_2, \ldots, e_{d + 1} \}) = \{ F(e_1), F(e_2), \ldots, F(e_{d + 1}) \}.$$

\begin{lem}[Structural Lemma]\label{lem:structure}
Let $a \in \left( \tfrac{1}{d + 1}, a_0 \right)$. Then for every Euclidean $d$-space $V$ one can define a continuous map
$$T^V_a : \mathcal M_a^{\circ}(V) \to \mathcal T(V)$$
such that
\begin{enumerate}
  \item $T^V_a$ is continuous.
  \item For any two Euclidean $d$-spaces $V_1$ and $V_2$ and any linear isometry $F : V_1 \to V_2$ the following diagram is commutative:
  $$
    \begin{tikzcd}
        \mathcal M_a^{\circ}(V_1) \arrow{r}{F_*} \arrow[swap]{d}{T^{V_1}_a} & \mathcal M_a^{\circ}(V_2) \arrow{d}{T^{V_2}_a} \\
        \mathcal T(V_1) \arrow{r}{F} & \mathcal T(V_2)
    \end{tikzcd}.
  $$
\end{enumerate}
\end{lem}

The intuition behind the Structural Lemma can be roughly explained considering, in some sense, a ``typical'' representative
of $\mathcal M_a^{\circ}(V)$ for some $a \in \left( \tfrac{1}{d + 1}, a_0 \right)$. Let $\{ e_1, e_2, \ldots, e_{d + 1} \} \in \mathcal T(V)$, and let
$$\nu = \frac{1}{d + 1}(\nu_1 + \nu_2 + \ldots + \nu_{d + 1}),$$
where $\nu_i$ is a nice measure sharply concentrated around $e_i$. (We also require $o(\nu) = \mathbf 0$, but this can also be
settled by the particular choice of $\nu_i$.) It is not hard to check that $\depth_{\nu}(0)$ is close to $\tfrac{1}{d + 1}$,
so, in particular, $\nu \in \mathcal M_a^{\circ}(V)$. If we were restricted only to this type of measures, then it would
have been natural to put
$$T^V_a(\nu) = \{ e_1, e_2, \ldots, e_{d + 1} \}.$$

Our goal will be to formalize this intuition showing that every $\nu \in \mathcal M_a^{\circ}(V)$ behaves, in a certain sense,
similarly to the described measures.

\section{Topological part}\label{subs:topol-overview}

We write $\mathbb RP^d$ for the space of all one-dimensional directions in $\mathbb R^{d + 1}$ as this is indeed the real projective space.

Let $\xi = (E, \mathbb RP^d, p)$ be the tautological quotient bundle~\cite[\S 2.2.3]{3264} over $\mathbb RP^d$. I.e., the total space $E$ can be written
as a quotient space
$$E = \{ (u, v) : u \in \mathbb S^d, v \in \mathbb R^{d + 1}, \langle u, v \rangle = 0 \} / \sim,$$
where the equivalence relation $\sim$ is defined by $(u, v) \sim (-u, v)$, and the projection $p: E \to \mathbb RP^d$ is as follows:
$$p(u, v) = \ell \Leftrightarrow \ell \parallel u.$$

There is a natural way to identify the fiber $p^{-1}(\ell)$ and the hyperplane $\ell^{\bot}$: a point $(u, v) \sim (-u, v)$, where $u \parallel \ell$
is identified with the point $v \in \ell^{\bot}$. (The last inclusion is due to the property $\langle u, v \rangle = 0$.)

Let us state and prove the key lemma of the topological part. After presenting the formal proof we will also give its less formal
(but also non-rigorous) interpretation.

The term {\it $k$-fold covering of $\mathbb RP^d$} will refer to a projection $p: X \to \mathbb RP^d$ appearing in the common definition
of a covering space (see, for example,~\cite[\S 1.3]{Hat}), and $k$ denotes the cardinality of each set $p^{-1}(\ell)$, where $\ell$ runs through
$\mathbb RP^d$.

\begin{lem}\label{lem:topological}
Let $d \geq 2$, $\xi = (E, \mathbb RP^d, p)$ be the tautological quotient bundle as above. Then there is no
space $E' \subset E$ such that
\begin{enumerate}
	\item The projection $p \mid_{E'}$ is a $(d + 1)$-fold covering of $\mathbb RP^d$ by $E'$.
	\item For every $\ell \in \mathbb RP^d$ one has $E' \cap p^{-1}(\ell) \in \mathcal T(p^{-1}(\ell))$.
\end{enumerate}
\end{lem}

\begin{proof}
Suppose that such $E'$ exists. Since $\pi_1(\mathbb RP^d) = \mathbb Z_2$, then $E'$ splits into 1-fold and 2-fold subcovers.

We will show then that in each case $\xi$ admits a non-vanishing section. That is, there exists $F \subset E$ such that
$p \mid_F$ is a homeomorphism from $F$ to $\mathbb RP^d$, and for every $\ell \in \mathbb RP^d$ the point $F \cap p^{-1}(\ell)$ is
not the origin of the fiber $p^{-1}(\ell)$.

If there is a 1-fold subcover $F \subset E$, then it is a non-vanishing section itself.
If there is a 2-fold subcover $G \subset E'$, then for each $\ell \in \mathbb RP^d$
define
$$f(\ell) = g_1 + g_2, \qquad \text{where $\{g_1, g_2\} = G \cap p^{-1}(\ell)$.}$$
(The sum of vectors is well defined in the fiber $p^{-1}(\ell)$.) Put
$$F = \{f(\ell) : \ell \in \mathbb RP^d \}.$$
Since $d \geq 2$, and $g_1$ and $g_2$ are two vertices of the simplex $S(\ell)$, which contains the origin inside, the sum $g_1 + g_2$
cannot vanish. Hence $F$ is a non-vanishing section.

Consider the cohomology ring $H^*(\mathbb RP^d, \mathbb Z_2)$. We have
$$H^*(\mathbb RP^d, \mathbb Z_2) = \mathbb Z_2[x] / (x^{d + 1}),$$
where $x \in H^1(\mathbb RP^d, \mathbb Z_2)$ (see~\cite[Lemma 4.3]{MSt}).
If $sw(\xi) \in H^*(\mathbb RP^d, \mathbb Z_2)$ is the Stiefel-Whitney class of the bundle $\xi$, then we have
$sw(\xi) = 1 + x + \ldots + x^d$ (see~\cite[\S 4, Example 3]{MSt}).

Hence the top ($d$-th) Stiefel-Whitney class of $\xi$ is non-zero, and, consequently, $\xi$ cannot have a non-vanishing section
(see~\cite[\S~12]{MSt} and references therein). A contradiction finishes the proof.

\end{proof}

\noindent {\it Remark (non-rigorous ``proof'' of Lemma~\ref{lem:topological}).}
We reduce the lemma to non-existence of a non-vanishing section of $\xi$ just as above. Assume, for a contradiction,
that $\xi$ has a non-vanishing section. Equivalently, there exists an {\it even} non-vanishing tangent vector field $v$ for
the sphere $\mathbb S^{d}$ (i.e., $v(-x) = v(x)$).

Let us also allow $v$ to have isolated finite-index singularities. An example
is the field $v_0$ of unit vectors pointing nothwards. The singularities of $v_0$ are the north and the south poles.
One can see that the index of $v_0$ at the poles is $\pm 1$. Hence the points with odd index split into an odd
number of pairs (each pair consists of two mutually antipodal points). We claim that this property holds not only for $v_0$,
but also for all possible vector fields.

Let $\mathcal C$ be a centrally symmetric simplicial subdivision of $\mathbb S^d$ isomorphic to the $(d + 1)$-crosspolytope
such that the poles of $\mathbb S^d$ do not belong to the $(d - 1)$-skeleton of $\mathcal C$.

Let $\Sigma$ be a simplex of $\mathcal C$, and $u_{\Sigma}$ be a vector field on $\partial \Sigma$. We assume only that
$u_{\Sigma}$ is tangent to $\mathbb S^d$ and has no singularities. Then all continuations of $u_{\Sigma}$ onto $\Sigma$
have the same sum of indices over all singularities. Denote this sum by $s(\Sigma, u_{\Sigma})$.

If a vector field $u$ is defined on $sk_{d - 1}(\mathcal C)$, we write
$$s(u) = \sum\limits_{(\Sigma, -\Sigma)} s(\Sigma, u \mid_{\partial \Sigma}).$$
(In the summation above we account for each pair $(\Sigma, -\Sigma)$ of antipodal simplices exactly once; the order in which the
pair is accounted for does not affect the parity of $s(u)$, because $s(\Sigma, u \mid_{\partial \Sigma}) \equiv s(-\Sigma, u \mid_{\partial (-\Sigma)}) \pmod 2$
for every even field $u$.) We want to prove that $s(u)$ has to be odd.

We can assume that $u$ coincides with $v_0$ on $sk_{d - 2}(\mathcal C)$; this can be done by a continuous perturbation of $u$.

Let us notice that, whenever a change of $u$ affects only a pair of antipodal $(d - 1)$-faces of $\mathcal C$, the parity of
$s(u)$ remains unchanged. Indeed, if the simplices $\Sigma$ and $\Sigma'$ share one of the chosen faces, then
the two terms,
$$s(\Sigma, u \mid_{\partial \Sigma})\quad \text{and} \quad s(\Sigma', u \mid_{\partial \Sigma'})$$
may either change or not change their parity simultaneously, and all the other terms remain unaffected. Hence, identifying $u$ with
$v_0$ face-by-face, we indeed get
$$s(u) \equiv s(v_0 \mid_{sk_{d - 1}(\mathcal C)}) \equiv 1 \pmod 2.$$
\qed

\noindent {\it Remark.}
The present version of Lemma~\ref{lem:topological} was suggested by R.~Karasev and works for every dimension.
For all dimensions, except $d = 3$ and $d = 7$, one can consider the tangent bundle of $\mathbb S^d$
with an embedded $(d + 1)$-fold cover. This gives $d + 1$ affinely independent vector fields in $\mathbb S^d$
(and therefore $d$ linearly independent vector fields), which is also impossible.

\section{Deduction of Theorem~\ref{thm:main} from Lemmas~\ref{lem:median}, \ref{lem:structure} and \ref{lem:topological}}\label{subs:proof}

Let $\mu$ be a nice measure in $\mathbb R^{d + 1}$. If $\xi = (E, \mathbb RP^d, p)$ is the tautological quotient bundle over $\mathbb RP^d$
as in previous section, we can think of a projected measure $\mu_{\ell}$ as a measure in $p^{-1}(\ell)$.

We will argue by contradiction. If Theorem~\ref{thm:main} is false, then there exists a measure $\mu \in \mathcal M(\mathbb R^{d + 1})$
such that for every $\ell \in \mathbb RP^d$ one has $\mu_{\ell} \in \mathcal M_{a_0}(p^{-1}(\ell))$.

Let us introduce the notation that we will refer to as {\it the local picture at $\ell_0$}.

Let $\ell_0$ be an arbitrary point of $\mathbb RP^d$. Then we can choose a neighborhood $U \subset \mathbb RP^d$ of $\ell_0$ and a homeomorphism
$$\phi : p^{-1}(U) \to U \times V,$$
where $V$ is a Euclidean $d$-space, such that the restriction $\phi\mid_{p^{-1}(\ell)}$ is a linear isometry of the spaces $p^{-1}(\ell)$ and $\{ \ell \} \times V$.
(We set the origin of $\{ \ell \} \times V$ to be the point $(\ell, \mathbf 0)$, where $\mathbf 0$ is the origin of $V$.)
Let $\pi : U \times V \to V$ be a projection preserving the $V$-component.

In the local picture at $\ell_0$ define
$$\nu_\ell = (\pi \circ \phi\mid_{p^{-1}(\ell)} )_*(\mu_{\ell}).$$
It is clear that $\nu(\ell) \in \mathcal M_{a_0}(V)$. Also, assertions 1--3 of Proposition~\ref{prop:nice_mes} imply that $\nu_\ell$
depends continuously on $\ell$. By Lemma~\ref{lem:median}, the Tukey median of $\nu_\ell$ is unique and continuous, so the measure $(\nu_\ell)^{\circ}$
depends continuously on $\ell$ as well. Finally, it is evident that
$$(\nu_\ell)^{\circ} = (\pi \circ \phi\mid_{p^{-1}(\ell)} )_*(\mu^{\circ}_{\ell}).$$

Write $a(\ell) = \depth_{(\mu_{\ell})^{\circ}}(\mathbf 0_{p^{-1}(\ell)})$. In the local picture at $\ell_0$ we have
$$a(\ell) = \depth_{(\nu_{\ell})^{\circ}}(\mathbf 0),$$
where $\mathbf 0$ is the origin of $V$. Thus, by Lemma~\ref{lem:median} and the continuous dependence of $\nu_{\ell}$ on $\ell$,
we conclude that $a(\ell)$ depends continuously on $\ell$ in $U$. In particular, $a(\ell)$ is continuous at the point $\ell = \ell_0$.
But $\ell_0$ is arbitrary, hence the continuity of $a(\ell)$ in the entire $\mathbb RP^d$ follows.

By compactness of $\mathbb RP^d$, the function $a(\ell)$ attains its maximum, so
$\sup\limits_{\ell \in \mathbb RP^d} a(\ell) < a_0$. Therefore one can choose $a_1 < a_0$ such that
$(\mu_{\ell})^{\circ} \in \mathcal M_{a_1}^{\circ}(p^{-1}(\ell))$ for every $\ell \in \mathbb RP^d$.

Define a space $E' \subset E$, using Lemma~\ref{lem:structure}, as follows:
$$E' = \bigcup \limits_{\ell \in \mathbb RP^d} T^{p^{-1}(\ell)}_{a_1}((\mu_{\ell})^{\circ}).$$
(Recall that we use the notation $\xi = (E, \mathbb RP^d, p)$ for the tautological quotient bundle over $\mathbb RP^d$.)

We claim that the $(d + 1)$-tuple of points $E' \cap p^{-1}(\ell)$ ($ = T^{p^{-1}(\ell)}_{a_1}((\mu_{\ell})^{\circ})$) depends continuously on $\ell$.

Consider the local picture at $\ell_0$. Since $\phi$ is a homeomorphism and $\ell_0$ is arbitrary, it will be sufficient to prove that
$$\phi\left( T^{p^{-1}(\ell)}_{a_1}((\mu_{\ell})^{\circ}) \right)$$
depends continuously on $\ell$ at $\ell = \ell_0$. Since the $U$-coordinate of all points in the above expression is $\ell$ (and so depends continuously on $\ell$),
we can ignore it. Thus, effectively, we need to prove the continuity of
$$(\pi \circ \phi)\left( T^{p^{-1}(\ell)}_{a_1}((\mu_{\ell})^{\circ}) \right).$$

By condition 2 of Lemma~\ref{lem:structure}, we have
\begin{multline*}
(\pi \circ \phi)\left( T^{p^{-1}(\ell)}_{a_1}((\mu_{\ell})^{\circ}) \right) = (\pi \circ \phi\mid_{p^{-1}(\ell)}) \left( T^{p^{-1}(\ell)}_{a_1}((\mu_{\ell})^{\circ}) \right) = \\
T^V_{a_1}((\pi \circ \phi\mid_{p^{-1}(\ell)})_*(\mu_{\ell})^{\circ})) = T^V_{a_1}(\nu_{\ell}).
\end{multline*}
But $\nu_{\ell}$ depends continuously on $\ell$, and, by condition 1 of Lemma~\ref{lem:structure}, $T^V_{a_1}$ is continuous.
Therefore $T^V_{a_1}(\nu_{\ell})$ indeed depends continuously on $\ell$. The claim is proved.

Consequently, the projection $p \mid_{E'}$ is a $(d + 1)$-fold covering of $\mathbb RP^d$ by $E'$. Moreover, by construction,
$E' \cap p^{-1}(\ell) \in \mathcal T(p^{-1}(\ell))$. Hence $E'$ satisfies the conditions of Lemma~\ref{lem:topological}, which
is impossible. The contradiction finishes the proof of Theorem~\ref{thm:main}.

\section{The Tukey median of a measure}\label{sect:centerpoint}

This entire section is devoted to the proof of Lemma~\ref{lem:median}.

If $n \in \mathbb R^d$ is a unit vector, denote by $H(n)$ the half-space such that the origin $\mathbf 0$ belongs to $\partial H(n)$
and $n$ is the outer normal to $\partial H(n)$, i.e., $n$ is orthogonal to  $\partial H(n)$ and is directed outwards the half-space $H(n)$.

\noindent {\bf Assertion 1.} Without loss of generality assume that $o$ coincides with the origin $\mathbf 0$.

Define $\mathcal N$ to be the set of all unit vectors $n \in \mathbb R^d$ such that $\nu(H(n)) = \depth_{\nu}(\mathbf 0)$. Clearly, the set
$\mathcal N$ is compact.

Assume that $\mathbf 0 \notin \conv \mathcal N$. Then $\conv \mathcal N$ can be separated from
$\mathbf 0$ by a hyperplane. Or, equivalently, there exists a unit vector $v$ such that
$$\inf\limits_{n \in \mathcal N} \left\langle n, v \right\rangle > 0.$$

This is impossible, since for small enough $\delta > 0$ we have $\depth(\delta \cdot v) > a$. Indeed,
if we translate $\nu$ by a vector $-\delta \cdot v$, then $\depth(\mathbf 0)$ increases. The reason is that the translation of $\nu$
by $-\delta \cdot v$ increases the measure of each $H(n)$ for all $n$ in some open neighborhood of $\mathcal N$ and does not sufficiently
decrease the measure of all other half-spaces with $\mathbf 0$ in the boundary. The contradiction shows that
$\mathbf 0 \in \conv \mathcal N$.

The Carath\'eodory Theorem implies that
$$\mathbf 0 \in \conv \{ n_1, n_2, \ldots, n_k \} \qquad \text{($n_i \in \mathcal N$, $2 \leq k \leq d + 1$).}$$
By the choice of $n_i$ we have
$$\nu \left( \bigcap\limits_{i = 1}^{k} H(-n_i) \right) = 0,$$
because the intersection is a $(d + 1 - k)$-dimensional affine plane through $\mathbf 0$. Equivalently,
$$\nu \left( \bigcup\limits_{i = 1}^{k} H(n_i) \right) = 1.$$
But $\nu(H(n_i)) = 1 - \nu(H(-n_i)) = \depth_{\nu}(\mathbf 0) < a$, so
$$ka > \sum\limits_{i = 1}^k \nu(H(n_i)) \geq \nu \left( \bigcup\limits_{i = 1}^{k} H(n_i) \right) = 1.$$
Thus $k > \tfrac{1}{a} > d$. This leaves the only option $k = d + 1$.

Therefore we have obtained a $(d + 1)$-tuple of half-spaces
$$(H(-n_1), H(-n_2), \ldots, H(-n_{d + 1}))$$
satisfying the requirements of Assertion 1.

\noindent {\bf Assertion 2.}  Without loss of generality assume that $\mathbf 0$ and $o \neq \mathbf 0$ are two different Tukey medians of $\nu$.
Since $\mathbf 0$ is a Tukey median of $\nu$, we can choose a $(d + 1)$-tuple of half-spaces
$$(H(-n_1), H(-n_2), \ldots, H(-n_{d + 1}))$$
as in the proof of Assertion 1.
Then for some $i$ we have $o \in \relint H(n_i)$.

Consider the half-space $H$ such that
$$o \in \partial H, \quad \partial H \bot n_i, \quad \text{and} \quad H \subset H(n_i).$$
By construction, $\mu(H) < \mu(H(n_i)) = \depth_{\nu}(\mathbf 0)$. Hence $\depth_{\nu}(o) < \depth_{\nu}(\mathbf 0)$.
But we assumed that $o$ is a Tukey median, i.e., $\depth_{\nu}(o) = \sup\limits_{x \in V} \depth_{\nu}(x) = \depth_{\nu}(\mathbf 0)$,
a contradiction.

\noindent {\bf Assertion 3.} For every $\nu, \nu' \in \mathcal M_a(V)$ and their respective Tukey medians $o, o'$ we have
\begin{multline}\label{eq:2}
\depth_{\nu}(o) - \| \nu - \nu' \| \leq \depth_{\nu'}(o) \leq \depth_{\nu'}(o') \leq \\
\depth_{\nu}(o') + \| \nu - \nu' \| \leq \depth_{\nu}(o) + \| \nu - \nu' \|.
\end{multline}
Hence Assertion 3 follows.

\noindent {\bf Assertion 4.} Let $\nu \in \mathcal M_a(V)$. Denote $o = o(\nu)$.

To prove the assertion, it is enough to prove the following claim: given an arbitrary neighborhood $U$ of $o$ there exists $\varepsilon > 0$
such that $o(\nu') \in U$ whenever $\nu' \in \mathcal M_a(V)$ and $\| \nu - \nu' \| < \varepsilon$.

Choose a $(d + 1)$-tuple of half-spaces
$$(H_1, H_2, \ldots, H_{d + 1})$$
satisfying the requirements of Assertion 1.

Next, choose a $(d + 1)$-tuple of half-spaces
$$(H'_1, H'_2, \ldots, H'_{d + 1})$$
such that
$$\partial H_i \parallel \partial H'_i, \quad H_i \subset \relint H'_i, \quad \text{and} \quad
S = \bigcap\limits_{i = 1}^{d + 1} H'_i \subset U.$$
By construction, $S$ is a $d$-simplex, and $o \in \relint S$.

Denote
$$\delta = \min\limits_{1 \leq i \leq d + 1} \nu(H'_i) - \nu(H_i).$$
Since $\nu$ is nice, $\delta > 0$. We will show that $\varepsilon = \delta / 2$ is sufficient.

Indeed, let $\nu'$ satisfy the assumptions of our claim. Assume, for a contradiction, that $o(\nu') \notin S$.
Then for some $i$ we have $o(\nu') \in \mathbb R^d \setminus H'_i$. Therefore
\begin{multline*}
\depth_{\nu'}(o(\nu')) \leq \nu'(\mathbb R^d \setminus H'_i) \leq \nu(\mathbb R^d \setminus H'_i) + \| \nu - \nu' \|
\leq \nu(\mathbb R^d \setminus H_i) - \delta + \| \nu - \nu' \| \\
< \depth_{\nu}(\mathbf 0) - \delta + \frac{\delta}{2} = \depth_{\nu}(\mathbf 0) - \frac{\delta}{2}.
\end{multline*}

On the other hand, according to \eqref{eq:2}, $\depth_{\nu'}(o(\nu')) > \depth_{\nu}(\mathbf 0) - \delta / 2$.
A contradiction shows that $o(\nu') \in S \subset U$. Thus our claim and Assertion 4 are proved.

\section{Generating $(d + 1)$-tuples of half-spaces}\label{sect:mainlemma}

Let $V$ be a Euclidean $d$-space. A $(d + 1)$-tuple of closed half-spaces
$$(H_1, H_2, \ldots, H_{d + 1}), \quad H_i \subset V$$
will be called {\it generating} if it satisfies $\bigcap\limits_{i = 1}^{d + 1} H_i = \{ \mathbf 0 \}$.
(See also the two-dimensional illustration, Figure~\ref{fig:1}.)

This definition implies, in particular, $\mathbf 0 \in \partial H_i$ for $i = 1, 2, \ldots, d + 1$. Another equivalent
definition would be as follows: a $(d + 1)$-tuple of closed half-spaces
$$(H_1, H_2, \ldots, H_{d + 1}), \quad H_i \subset V$$
is generating if for every $i$ one has $\mathbf 0 \in \partial H_i$ and
$$V \setminus  \mathbf 0 = \bigcup\limits_{i = 1}^{d + 1} \relint H_i.$$
Note that Lemma~\ref{lem:median}, assertion 1 in the case $o = \mathbf 0$ claims exactly the existence of a
generating $(d + 1)$-tuple with a certain property.

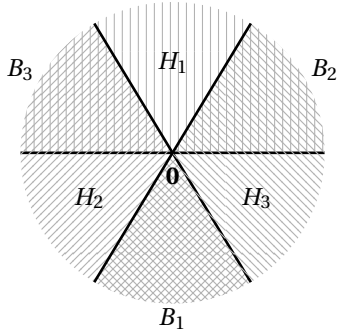
\begin{figure}
\begin{tikzpicture}[line width = 1pt]
\begin{scope}
\clip (0, 0) circle (2cm);
\filldraw[pattern=vertical lines, pattern color=black!30] (-3, 3) -- (-3, 0) -- (3, 0) -- (3, 3);
\filldraw[pattern=north east lines, pattern color=black!30] (-8, 0) -- (-3, 5) -- (3, -5) -- (0, -8);
\filldraw[pattern=north west lines, pattern color=black!30] (8, 0) -- (3, 5) -- (-3, -5) -- (0, -8);
\end{scope}
\node at (0, 1.2) {$H_1$};
\node at (1.1, -0.6) {$H_3$};
\node at (-1.1, -0.6) {$H_2$};
\node at (0, -2.2) {$B_1$};
\node at (-2.0, 1.1) {$B_3$};
\node at (2.0, 1.1) {$B_2$};
\node at (0, -0.3) {$\mathbf 0$};
\end{tikzpicture}
\caption{A generating 3-tuple in the plane.}\label{fig:1}
\end{figure}

In this section we prove several auxiliary facts concerning generating $(d + 1)$-tuples of cones.

For a generating $(d + 1)$-tuple $(H_1, H_2, \ldots, H_{d + 1})$ define the {\it corresponding} $(d + 1)$-tuple of simplicial cones
$(B_1, B_2, \ldots, B_{d + 1})$ by
\begin{equation}\label{eq:restricting-cones-def}
B_i = \bigcap\limits_{\substack{1 \leq j \leq d + 1\\ j \neq i}} H_j.
\end{equation}
Clearly, different $(d + 1)$-tuples of half-spaces generate different $(d + 1)$-tuples of simplicial cones.

\begin{lem}\label{lem:interior}
Let $(H_1, H_2, \ldots, H_{d + 1})$ be a generating $(d + 1)$-tuple of half-spaces in the Euclidean $d$-space $V$,
$(B_1, B_2, \ldots, B_{d + 1})$ --- the corresponding $(d + 1)$-tuple of simplicial cones.
Assume that a point $b_i \in \relint B_i$ be chosen for each $i = 1, 2, \ldots, d + 1$.
Then $\conv \{ b_1, b_2, \ldots, b_{d + 1} \}$ is a non-degenerate $d$-simplex, and
$$\mathbf 0 \in \relint \conv \{ b_1, b_2, \ldots, b_{d + 1} \}.$$
\end{lem}

\begin{proof}
For each plane $\partial H_i$ denote by $n_i$ the unit normal vector directed outwards $H_i$. We have
\begin{equation}\label{eq:convhull}
\mathbf 0 \in \relint \conv \{n_1, n_2, \ldots, n_{d + 1} \},
\end{equation}
otherwise
$$\bigcap\limits_{i = 1}^{d + 1} H_i \neq \{ \mathbf 0 \}.$$

By construction, for every $i \neq j$ we have
$$\left\langle b_i, n_i \right\rangle > 0, \quad \left\langle b_i, n_j \right\rangle < 0.$$

Now we argue by contradiction. Assuming that the statement of lemma is false, there is a plane $\alpha$ that separates $\mathbf 0$
from every $b_i$. (The separation need not be strict.) If $n$ is the normal vector to $\alpha$ pointing towards the open
half-space with all $b_i$, then
$$\left\langle b_i, n \right\rangle \geq 0 \quad \text{for all $i$.}$$

Due to~\eqref{eq:convhull}, we may assume without loss of generality that
$$n = \lambda_1 n_1 + \lambda_2 n_2 + \ldots + \lambda_d n_d,$$
where $\lambda_i \geq 0$ and not all $\lambda_i$ are zero. Then
$$\left\langle b_{d + 1}, n \right\rangle =
\lambda_1\left\langle b_{d + 1}, n_1 \right\rangle + \lambda_2\left\langle b_{d + 1}, n_2 \right\rangle + \ldots +
\lambda_d\left\langle b_{d + 1}, n_d \right\rangle < 0,$$
a contradiction.

\end{proof}

A generating $(d + 1)$-tuple $(H_1, H_2, \ldots, H_{d + 1})$ of half-spaces is said to have {\it weight}
$a$ with respect to a measure $\nu$, if
$$\min\limits_i \nu(H_i) = 1 - a.$$
Then we will write
$$\wgt_{\nu}(H_1, H_2, \ldots, H_{d + 1}) = a \quad \text{or} \quad \wgt_{\nu}(n_1, n_2, \ldots, n_{d + 1}) = a,$$
where $n_i$ is an outer normal for $H_i$.

\begin{lem}\label{lem:b-mes}
Let $V$ be a Euclidean $d$-space, $\nu \in \mathcal M(V)$, $\varepsilon \in \left(0, \tfrac{1}{(d + 1)(2d + 1)} \right]$.
Assume that a generating $(d + 1)$-tuple $(H_1, H_2, \ldots, H_{d + 1})$ of half-spaces in $V$ satisfies
$$\wgt_{\nu}(H_1, H_2, \ldots, H_{d + 1}) < \frac{1}{d + 1} + \varepsilon.$$
Then the corresponding $(d + 1)$-tuple of simplicial cones $(B_1, B_2, \ldots, B_{d + 1})$ satisfies
\begin{enumerate}
	\item $\sum\limits_{i = 1}^{d + 1} \nu(B_i) > 1 - (d + 1)\varepsilon$.
	\item $\frac{1}{d + 1} - (2d + 1)\varepsilon < \nu(B_i) < \frac{1}{d + 1} + \varepsilon$.
\end{enumerate}
\end{lem}

\begin{proof}

One can see that the set $\bigcup\limits_{i = 1}^{d + 1} B_i$ is exactly the region that is covered exactly $d$ times by
the covering family of half-spaces $H_i$. In turn, the set $V \setminus \bigcup\limits_{i = 1}^{d + 1} B_i$ is covered at most $d - 1$ times. Hence
$$d \sum\limits_{i = 1}^{d + 1} \nu(B_i) + (d - 1) \left(1 - \sum\limits_{i = 1}^{d + 1} \nu(B_i) \right) > (d + 1)\left( 1 - \frac{1}{d + 1} - \varepsilon \right),$$
or
$$(d - 1) + \sum\limits_{i = 1}^{d + 1} \nu(B_i) > d - (d + 1)\varepsilon.$$
This proves Assertion 1.

Without loss of generality let
$$\nu(B_1) \leq \nu(B_2) \leq \ldots \leq \nu(B_{d + 1}).$$

Assume that the second inequality in Assertion 2 is false, and $\nu(B_{d + 1}) \geq \tfrac{1}{d + 1} + \varepsilon$. Then
$$\nu(H_{d + 1}) < 1 - \nu(B_{d + 1}) \leq 1 - \frac{1}{d + 1} - \varepsilon.$$
Consequently, $\wgt_{\nu}(H_1, H_2, \ldots, H_{d + 1}) > \tfrac{1}{d + 1} + \varepsilon$, a contradiction.

Assume that the first inequality in Assertion 2 is false, and $\nu(B_1) \leq \tfrac{1}{d + 1} - (2d + 1)\varepsilon$. Then, by Assertion 1,
$$\sum\limits_{i = 2}^{d + 1} \nu(B_i) \geq \frac{d}{d + 1} + d\varepsilon.$$
Thus $\nu(B_{d + 1}) \geq \tfrac{1}{d + 1} + \varepsilon$, contradicting the second inequality of Assertion 2, which has already been proved.
Therefore Assertion 2 is proved completely.

\end{proof}

\begin{lem}[Bijection Lemma]\label{lem:main}
Let $V$ be a Euclidean $d$-space, $\nu \in \mathcal M(V)$, $\varepsilon \in \left(0, \tfrac{1}{(d + 1)(3d + 2)} \right]$.
Assume that there are two generating $(d + 1)$-tuples of half-spaces in $V$,
$$(H_1, H_2, \ldots, H_{d + 1}), \quad (H'_1, H'_2, \ldots, H'_{d + 1}),$$
satisfying
$$\wgt_{\nu}(H_1, H_2, \ldots, H_{d + 1}), \wgt_{\nu}(H'_1, H'_2, \ldots, H'_{d + 1}) < \frac{1}{d + 1} + \varepsilon.$$
Then the corresponding $(d + 1)$-tuples of simplicial cones
$$(B_1, B_2, \ldots, B_{d + 1}), \quad (B_1, B_2, \ldots, B_{d + 1})$$
satisfy
$$\nu\left( B_i \cap B'_{\sigma(i)} \right) > \frac{1}{d + 1} - (3d + 2) \varepsilon \qquad
\nu\left( B_i \cap B'_{\sigma(j)} \right) = 0$$
for every $i, j \in \{1, 2, \ldots, d + 1\}$, $i \neq j$ and some permutation $\sigma$ of the set $\{1, 2, \ldots, d + 1\}$.
\end{lem}

\begin{proof}
Consider the bipartite graph $G$ whose vertex set is $\{ B_1, B_2, \ldots, B_{d + 1} \} \cup \{ B'_1, B'_2, \ldots, B'_{d + 1} \}$
(the parts are $\{ B_1, B_2, \ldots, B_{d + 1} \}$ and $\{ B'_1, B'_2, \ldots, B'_{d + 1} \}$), and an edge $(B_i, B'_j)$
is present in $G$ if and only if $\nu(B_i \cap B'_j) > 0$.

We claim that $G$ is a perfect matching. First we prove that $G$ contains a perfect matching as a subgraph.

Assume that there is no perfect matching in $G$. Then, up to a permutation of indices,
there exist $k, m \in \mathbb N$, $k < m$, such that $B_1, B_2, \ldots, B_m$ are connected only with $B'_1, B'_2, \ldots, B'_k$.
Indeed, this is a direct consequence of the Hall's Marriage Theorem~\cite[Theorem 25.1]{wil}.

Thus the set $\bigcup\limits_{i = 1}^m B_i$ is covered by the set
$$\bigcup\limits_{i = 1}^k B'_i \cup \left( V \setminus \bigcup\limits_{i = 1}^{d + 1} B'_i \right).$$
Hence
$$
\bigcup\limits_{i = m + 1}^{d + 1} B_i \cup \left( V \setminus \bigcup\limits_{i = 1}^{d + 1} B_i \right) \cup
\bigcup\limits_{i = 1}^k B'_i \cup \left( V \setminus \bigcup\limits_{i = 1}^{d + 1} B'_i \right) = V,
$$
and, consequently,
$$
\sum\limits_{i = m + 1}^{d + 1} \nu(B_i) + \nu \left( V \setminus \bigcup\limits_{i = 1}^{d + 1} B_i \right) +
\sum\limits_{i = 1}^k \nu(B'_i) + \nu \left( V \setminus \bigcup\limits_{i = 1}^{d + 1} B'_i \right) \geq 1.
$$
As $\varepsilon < \tfrac{1}{(d + 1)(3d + 2)} < \frac{1}{(d + 1)(2d + 1)}$, it is possible to apply Lemma~\ref{lem:b-mes},
replacing $\nu(B_i)$, $\nu(B'_i)$, $\nu \left( V \setminus \bigcup\limits_{i = 1}^{d + 1} B_i \right)$ and
$\nu \left( V \setminus \bigcup\limits_{i = 1}^{d + 1} B'_i \right)$ with their respective upper bounds. This yields
$$(d + 1 + k - m)\left( \frac{1}{d + 1} + \varepsilon \right) + 2(d + 1)\varepsilon > 1,$$
or
$$\varepsilon > \frac{m - k}{(d + 1)(3d + 3 + k - m)} \geq \frac{1}{(d + 1)(3d + 2)},$$
a contradiction. Therefore $G$ contains a perfect matching.

Up to a permutation of indices we can assume that for each $i$ an edge $(B_i, B'_i)$ is present in $G$. Now we aim to show that
no other edge of $G$ exists.

Without loss of generality, assume that $(B_1, B'_2)$ is also an edge of $G$. Choose a point $b_1 \in \relint B_1 \cap B'_2$.
For each $i > 1$ choose a point $b_i \in B_i \cap B'_i$.

We have $b_i \in \relint B_i$ for each $i$.
Hence, by Lemma~\ref{lem:interior},
$$\mathbf 0 \in \relint \conv \{ b_1, b_2, \ldots, b_{d + 1} \}.$$
On the other hand,
$$\mathbf 0 \notin \relint H'_1 \supset \conv \{ b_1, b_2, \ldots, b_{d + 1} \},$$
a contradiction. Therefore $G$ is exactly a perfect matching.

Finally, according to Lemma~\ref{lem:b-mes}, $\nu(B_i) > \frac{1}{d + 1} - (2d + 1)\varepsilon$,
and
$$\nu(B_i \setminus B'_i) \leq \nu \left( V \setminus \bigcup\limits_{i = 1}^{d + 1} B'_i \right) \leq (d + 1)\varepsilon.$$
Hence
$$\nu(B_i \cap B'_i) > \frac{1}{d + 1} - (3d + 2)\varepsilon.$$

\end{proof}

\section{The central ray of a simplicial cone}\label{sect:centralray}

Let us assume for this entire section that $V$ is a Euclidean $d$-space, and a measure $\nu \in \mathcal M(V)$ is fixed.

Let $B$ be a $d$-dimensional simplicial cone in $V$ with vertex $\mathbf 0$. For $t \in (0, 1]$ define
$$\mathcal H(B, t) = \{H : \text{$H$ is a half-space, $\mathbf 0 \in \partial H$, and $\nu(H \cap B) \geq t\nu(B)$} \}.$$

The cone $B$ does not contain any straight line entirely. Therefore we can choose a unit vector $n$ such that
$\left\langle n, b \right\rangle > 0$ for every $b \in B \setminus \mathbf 0$.

Consider the central projection $\pi_c$ of $V$ with the center at $\mathbf 0$ onto the plane
$$\Pi = \{ y : \left\langle n, y \right\rangle = 1 \}.$$
Define the probability measure $\nu^*$ in the plane $\Pi$ as follows:
$$\nu^*(X) = \frac{\nu(\pi_c^{-1}(X) \cap B)}{\nu(B)}$$
for every measurable $X \subseteq \Pi$. This means, for instance, that the support of $\nu^*$ is $\pi_c(B)$,
so $\nu^* \notin \mathcal M(\Pi)$, but we will not need the inclusion.

Let $H$ be a half-space in $V$ such that $\partial H$ is not orthogonal to $n$.
Then $H \cap \Pi$ is a half-space in $\Pi$ and
$$\nu^*(H \cap \Pi) = \frac{\nu(H \cap B)}{\nu(B)}.$$
If $t \geq \tfrac{d - 1}{d}$, then, by the Rado theorem,
$$\bigcap\limits_{H \in \mathcal H(B, t)} (H \cap \Pi) \neq \varnothing.$$
The intersection above is contained in $B$, hence there exists a non-zero vector $e \in B$ such that
$$\bigcap\limits_{H \in \mathcal H(B, t)} (H \cap B) \supseteq \{ \lambda e: \lambda \geq 0 \}.$$

The above argument for $t = \tfrac{d}{d + 1} > \tfrac{d - 1}{d}$ implies that the set
$$C(B) = \bigcap\limits_{H \in \mathcal H \left(B, \frac{d}{d + 1} \right)} (H \cap B)$$
is a convex cone of positive measure. We will call $C(B)$ the {\it central cone} of $B$.

Let $\mathbb S^{d - 1}$ be the unit sphere in $V$ centered at the origin. Define
$$e(B) = \frac{\int\limits_{\mathbb S^{d - 1} \cap C(B)} x\, dx}{\left\| \int\limits_{\mathbb S^{d - 1} \cap C(B)} x\, dx \right\|},$$
where $dx$ is the element of the $(d - 1)$-dimensional Lebesgue measure in $\mathbb S^{d - 1}$. We will call $e(B)$ the {\it central vector} of
$B$, and the ray along $e(B)$ the {\it central ray} of $B$.

We emphasize that $e(B)$ is a unit vector, and that $e(B) \in C(B)$.

\noindent {\it Remark.} As one can see, the definition of $C(B)$ and $e(B)$ depends on the measure $\nu$. Whenever an ambiguity
related to the varying measure $\nu$ is possible, we will use the notation $C(B; \nu)$ and $e(B; \nu)$.

Let us state some properties of central cones and central vectors.

\begin{prop}\label{prop:central-vec-contunuous}
$C(B; \nu)$ and $e(B; \nu)$ change continuously with a continuous change of $\nu$ and $B$.
\end{prop}

\noindent {\it Remark.}
In order to make Proposition~\ref{prop:central-vec-contunuous} explicit, we need to define a continuous change of a cone.
For instance, it is enough to define a basic neighborhood of a convex
$d$-dimensional cone $C \subset V$ containing no entire straight line. Namely, choose an arbitrary pair of closed convex cones
$C_{int}$ and $C_{ext}$ such that
$$C_{int} \subset \relint C, \quad C \subset \relint C_{ext}.$$
Then $C_{int}$ and $C_{ext}$ span the following basic neighborhood of $C$: the set of all closed cones $C'$ satisfying
$$C_{int} \subset \relint C', \quad C' \subset \relint C_{ext}.$$

The proof of Proposition~\ref{prop:central-vec-contunuous} is routine, and we therefore skip it.

\begin{lem}\label{lem:central-cone}
Let $B$ and $B'$ be simplicial cones, both with vertex $\mathbf 0$. Suppose that
$$\max(\nu(B), \nu(B')) \leq \frac{1}{d + 1} + \frac{1}{3(d + 1)^3},$$
$$\nu(B \cap B') \geq \frac{1}{d + 1} - \frac{3d + 2}{3(d + 1)^3}.$$
Then
$$B \supseteq C(B') \quad \text{and} \quad B' \supseteq C(B).$$
\end{lem}

\begin{proof}
The conditions are symmetric for $B$ and $B'$, so it is enough to prove that $B \supseteq C(B')$.
Let $C(B') \setminus B \neq \varnothing$. Then there exists a half-space $H$ such that
$$\mathbf 0 \in \partial H, \quad B \subset H, \quad \text{and} \quad C(B') \setminus H \neq \varnothing.$$

But we have
$$\nu(H \cap B') \geq \nu(B \cap B') \geq \frac{1}{d + 1} - \frac{3d + 2}{3(d + 1)^3} \geq
\frac{d}{d + 1} \left( \frac{1}{d + 1} + \frac{1}{3(d + 1)^3} \right) \geq \frac{d}{d + 1} \nu(B').$$
Hence $C(B') \subset H$ by definition of $C(B')$. A contradiction.

\end{proof}

\noindent {\it Remark.}
This is the point of the paper where the strongest assumptions are made. Indeed, the lower bound for $\nu(B \cap B')$
will come from Lemma~\ref{lem:main}. To make this lower bound work in Lemma~\ref{lem:central-cone}, we need
$\varepsilon \leq \tfrac{1}{3(d + 1)^3}$, which is much stronger than the assumption of Lemma~\ref{lem:main}.

\noindent {\it Remark.}
Notice that the value $t = \tfrac{d}{d + 1}$ for $\mathcal H(B, t)$ is not optimal. We could use any
$t > \tfrac{d - 1}{d}$, but decreasing $t$ gives only a minor improvement to our results.

\begin{cor}\label{cor:inwards}
If $B$, $B'$ are as in Lemma~\ref{lem:central-cone}, then
$$e(B) \in B' \quad \text{and} \quad e(B') \in B.$$
\end{cor}

\section{Ordered $(d + 1)$-tuples of small weight}\label{sect:order}

In this section we continue to write $V$ for a Euclidean $d$-space. For brevity, we write
$$a_0 = \frac{1}{d + 1} + \frac{1}{3(d + 1)^3}.$$

From now on, we start to distinguish ordered and unordered $(d + 1)$-tuples of half-spaces (cones). To emphasize the distinction,
we write unordered $(d + 1)$-tuples in circle brackets, and the ordered $(d + 1)$-tuples in square brackets.

Given a measure $\nu \in \mathcal M(V)$ and $a \in (0, 1)$, let $\mathcal R_{\nu}(a)$ denote the family of all (unordered) generating
$(d + 1)$-tuples of half-spaces $(H_1, H_2, \ldots, H_{d + 1})$ satisfying
$$\wgt_{\nu}(H_1, H_2, \ldots, H_{d + 1}) \leq a.$$

Assume, in addition, that $\nu \in \mathcal M^{\circ}_a(V)$ for some $a \in (0, a_0)$. Starting from the family $\mathcal R_{\nu}(a)$, we want
to obtain a family $\mathcal R_{\nu}^*(a)$ of ordered $(d + 1)$-tuples with the following natural properties.
\begin{enumerate}
	\item[(R1)] There is a bijection between $\mathcal R_{\nu}(a)$ and $\mathcal R^*_{\nu}(a)$: every unordered $(d + 1)$-tuple
				$(H_1, H_2, \ldots, H_{d + 1}) \in \mathcal R_{\nu}(a)$ corresponds to a unique ordered $(d + 1)$-tuple
                $$[H_{\sigma(1)}, H_{\sigma(2)}, \ldots, H_{\sigma(d + 1)}] \in \mathcal R^*_{\nu}(a),$$
                where $\sigma$ is some permutation of $\{ 1, 2, \ldots, d + 1 \}$.
	\item[(R2)] If
                $$[H_1, H_2, \ldots, H_{d + 1}], [H'_1, H'_2, \ldots, H'_{d + 1}] \in \mathcal R^*_{\nu}(a),$$
                $[B_1, B_2, \ldots, B_{d + 1}]$ and $[B'_1, B'_2, \ldots, B'_{d + 1}]$ are the corresponding $(d + 1)$-tuples of cones,
				then $\nu(B_i \cap B'_i) > \tfrac{1}{d + 1} - \tfrac{3d + 2}{3(d + 1)^3}$.
\end{enumerate}

\begin{lem}\label{lem:order-single}
For every $a \in \left( \tfrac{1}{d + 1}, a_0 \right)$ and every $\nu \in \mathcal M^{\circ}_{a}(V)$
there exists a family of ordered $(d + 1)$-tuples of half-spaces $\mathcal R_{\nu}^*(a)$ satisfying
the conditions (R1) and (R2).
\end{lem}

\begin{proof}
Choose an arbitrary unordered $(d + 1)$-tuple from $\mathcal R_{\nu}(a)$ and select an arbitrary order for it, say,
$$[H_1, H_2, \ldots, H_{d + 1}].$$

According to Lemma~\ref{lem:main}, for any unordered $(d + 1)$-tuple
$$(H'_1, H'_2, \ldots, H'_{d + 1}) \in \mathcal R_{\nu}(a)$$
there exists a permutation $\sigma$ of $\{ 1, 2, \ldots, d + 1 \}$ such that
\begin{equation}\label{eq:order_choice}
\nu \left( B_i \cap B'_{\sigma(i)} \right) \geq \frac{1}{d + 1} - \frac{3d + 2}{3(d + 1)^3} \qquad
\nu \left( B_i \cap B'_{\sigma(j)} \right) = 0,
\end{equation}
where $B_i$ and $B'_i$ denote the respective simplicial cones. Then put
$$[H'_{\sigma(1)}, H'_{\sigma(2)}, \ldots, H'_{\sigma(d + 1)}] \in \mathcal R^*_{\nu}(a).$$

Let
$$[H'_1, H'_2, \ldots, H'_{d + 1}], [H''_1, H''_2, \ldots, H''_{d + 1}] \in \mathcal R^*_{\nu}(a).$$
Then, combining~\eqref{eq:order_choice} and Lemma~\ref{lem:b-mes} for the cones $B_i$, we obtain
$$\nu \left( B_i \cap B'_i \right) > \frac{1}{2} \nu (B_i), \quad
\nu \left( B_i \cap B''_i \right) > \frac{1}{2} \nu (B_i).$$
Therefore
$$\nu \left( B'_i \cap B''_i \right) > 0.$$

Lemma~\ref{lem:main} implies
$$\nu \left( B'_i \cap B'_i \right) \geq \frac{1}{d + 1} - \frac{3d + 2}{3(d + 1)^3}.$$
Hence $\mathcal R^*_{\nu}(a)$ meets the required conditions.

\end{proof}

Let us also notice that the proof of Lemma~\ref{lem:order-single} implies the following proposition.

\begin{prop}
If a family $\mathcal R^*_{\nu}(a)$ satisfies the conditions (R1) and (R2), then
any other family satisfying these conditions is obtained by choosing an arbitrary permutation $\sigma$
and applying it to each element of $\mathcal R_{\nu}^*(a)$.
\end{prop}

\begin{lem}\label{lem:ordering-in-neighborhood}
Let $a \in \left( \tfrac{1}{d + 1}, a_0 \right)$,  $\nu \in \mathcal M^{\circ}_{a}(V)$, $a_1 \in (a, a_0)$. Assume that the family of ordered $(d + 1)$-tuples
$\mathcal R^*_{\nu}(a_1)$ is chosen. Then for every measure $\nu' \in \mathcal M^{\circ}_{a}(V)$ satisfying $\| \nu' - \nu \| < a_1 - a$
one can choose the family of ordered $(d + 1)$-tuples $\mathcal R^*_{\nu'}(a)$ satisfying (R1), (R2) and the additional condition
$$\mathcal R^*_{\nu'}(a) \subseteq \mathcal R^*_{\nu}(a_1).$$
\end{lem}

\begin{proof}
We start with the observation that $\mathcal R_{\nu'}(a) \neq \varnothing$. Indeed, this follows from
the assumption $\nu' \in \mathcal M^{\circ}_{a}(V)$ and assertion 1 of Lemma~\ref{lem:median}. Therefore we can choose a
$(d + 1)$-tuple $(H_1, H_2, \ldots, H_{d + 1}) \in \mathcal R_{\nu'}(a)$.

Since $\| \nu' - \nu \| < a_1 - a$, we conclude that
$$(H_1, H_2, \ldots, H_{d + 1}) \in \mathcal R_{\nu}(a_1).$$
Without loss of generality assume that $[H_1, H_2, \ldots, H_{d + 1}] \in \mathcal R^*_{\nu}(a_1)$.

Choose $\mathcal R^*_{\nu'}(a)$ so that $[H_1, H_2, \ldots, H_{d + 1}] \in \mathcal R^*_{\nu'}(a)$.

Assume that $[H'_1, H'_2, \ldots, H'_{d + 1}] \in \mathcal R^*_{\nu'}(a)$. Then
$$(H'_1, H'_2, \ldots, H'_{d + 1}) \in \mathcal R_{\nu}(a_1).$$
But we have $\nu'(B_i \cap B'_i) \geq \tfrac{1}{d + 1} - \tfrac{3d + 2}{3(d + 1)^3}$, and therefore
$$\nu(B_i \cap B'_i) \geq \frac{1}{d + 1} - \frac{3d + 2}{3(d + 1)^3} - (a_1 - a) > 0,$$
because $a_1 - a < a_0 - a < \tfrac{1}{3(d + 1)^3}$. Hence indeed $[H'_1, H'_2, \ldots, H'_{d + 1}] \in \mathcal R^*_{\nu}(a_1)$.

\end{proof}

To state the next lemma note that $\mathcal R^*_{\nu}(a)$ can be treated as a subset of the compact space $(\mathbb S^{d - 1})^{d + 1}$ with the natural
topology. Indeed, an ordered $(d + 1)$-tuple of half-spaces can be identified with the ordered $(d + 1)$-tuple of their outer unit normals.

\begin{lem}\label{lem:compact}
For every $a \in \left( \tfrac{1}{d + 1}, a_0 \right)$, $\nu \in \mathcal M^{\circ}_a(V)$ the set $\mathcal R^*_{\nu}(a)$ is compact.
\end{lem}

\begin{proof}
Consider an arbitrary converging sequence
$$[H_1^{(j)}, H_2^{(j)}, \ldots, H_{d + 1}^{(j)}] \in \mathcal R^*_{\nu}(a) \quad (j = 1, 2, \ldots).$$
Let $[B_1^{(j)}, B_2^{(j)}, \ldots, B_{d + 1}^{(j)}]$ be the respective $(d + 1)$-tuples of simplicial cones.
Write $H_i = \lim\limits_{j \to \infty} H_i^{(j)}$.

By property (R2) of $\mathcal R^*_{\nu}(a)$, we have
$$\nu(B^{(1)}_i \cap B^{(j)}_i) \geq \frac{1}{d + 1} - \frac{3d + 2}{3(d + 1)^3}.$$

Lemma~\ref{lem:central-cone} implies
$$C(B^{(1)}_i) \subseteq B^{(j)}_i.$$
Therefore, for each $1 \leq i, i' \leq d + 1$, $i \neq i'$ one has
$$C(B^{(1)}_i) \subset V \setminus H^{(j)}_i, \quad C(B^{(1)}_i) \subset H^{(j)}_{i'}.$$

For each $i = 1, 2, \ldots, d + 1$ take a unit vector $b_i$ pointing to the interior of $C(B^{(1)}_i)$.
If $H$ is a half-space with $\mathbf 0 \in \partial H$, let $n(H)$ denote the outer unit normal to $\partial H$.

As $b_i$ is separated from the boundary of $C(B_i)$, there exists $\delta > 0$, independent of $j$,
such that for any $1 \leq i, i' \leq d + 1$, $i \neq i'$ one has
$$\left\langle n(H^{(j)}_i), b_i \right\rangle \geq \delta, \quad
\left\langle n(H^{(j)}_{i'}), b_i \right\rangle \leq -\delta.$$
Taking the limit, we obtain
$$\left\langle n(H_i), b_i \right\rangle \geq \delta, \quad
\left\langle n(H_{i'}), b_i \right\rangle \leq -\delta.$$

By Lemma~\ref{lem:interior},
$$\mathbf 0 \in \relint \conv \{  b_1, b_2, \ldots, b_{d + 1} \}.$$
Hence, if $H^*_i$ is a half-space such that $b_i = n(H^*_i)$, then the $(d + 1)$-tuple
$$(H_1^*, H_2^*, \ldots, H_{d + 1}^*)$$
is generating. Let
$$(B_1^*, B_2^*, \ldots, B_{d + 1}^*)$$
be the corresponding $(d + 1)$-tuple of simplicial cones. Then one has
$$n(H_i) \in \relint B_i^*.$$
Lemma~\ref{lem:interior} immediately yields
$$\mathbf 0 \in \relint \conv \{  n(H_1), n(H_2), \ldots, n(H_{d + 1}) \}.$$
Hence the $(d + 1)$-tuple $[H_1, H_2, \ldots, H_{d + 1}]$ is generating.
Denote the corresponding simplicial cones by $B_i$ ($i = 1, 2, \ldots, d + 1$).

By continuity of the weight function,
$$\wgt_{\nu}(H_1, H_2, \ldots, H_{d + 1}) \leq a.$$

Finally, suppose that
$$[H_1, H_2, \ldots, H_{d + 1}] \notin \mathcal R^*_{\nu}(a).$$
Then there is a non-trivial permutation $\sigma$ such that
$$[H_{\sigma(1)}, H_{\sigma(2)}, \ldots, H_{\sigma(d + 1)}] \in \mathcal R^*_{\nu}(a).$$
Due to the property (R2), for each $j = 1, 2, \ldots$ we have
$$\nu(B_{\sigma(i)} \cap B^{(j)}_i) \geq \frac{1}{d + 1} - \frac{3d + 2}{3(d + 1)^3}.$$
Taking the limit for $j \to \infty$ yields
$$\nu(B_{\sigma(i)} \cap B_i) \geq \frac{1}{d + 1} - \frac{3d + 2}{3(d + 1)^3},$$
which is impossible. A contradiction shows that
$$[H_1, H_2, \ldots, H_{d + 1}] \in \mathcal R^*_{\nu}(a).$$

Thus $\mathcal R^*_{\nu}(a)$ is closed and hence compact.

\end{proof}

\section{Proof of Lemma~\ref{lem:structure}}\label{sect:cover}

We continue using the notation of the previous section. Also, throughout this section we will assume that
$a \in \left( \tfrac{1}{d + 1}, a_0 \right)$ is fixed. We prove Lemma~\ref{lem:continuous_n} to enable the definition
of $T_a^V$. Lemma~\ref{lem:interior2} shows that the image of $T_a^V$ is indeed a subset of $\mathcal T(V)$. Property 1
of Lemma~\ref{lem:structure} follows from Lemma~\ref{lem:aux-ei-continuous} and Corollary~\ref{cor:t-continuous}. Property 2
follows immediately from the definition, because all the auxiliary objects we use change naturally under isometries.

Let us define a vector function
$$e_i(\nu; n_1, n_2, \ldots, n_{d + 1}) : \mathcal M^{\circ}_a(V) \times ({\mathbb S}^{d - 1})^{d + 1} \to V.$$
I.e., the arguments are a measure $\nu \in \mathcal M^{\circ}_a(V)$ and $d + 1$ unit vectors in $V$.

In order to do that, choose, according to Lemma~\ref{lem:order-single}, the family $\mathcal R^*_{\nu}(a)$ of ordered $(d + 1)$-tuples of
half-spaces satisfying the properties (R1) and (R2) from the previous section.

Given a unit vector $n \in V$ denote by $H(n)$ the half-space in $V$ such that $\mathbf 0 \in \partial H(n)$ and $n$ is the
outer unit normal for $H(n)$. Let us also write $H_i = H(n_i)$. The definition of the function $e_i$ will consist of two mutually disjoint cases.

\noindent {\bf Case 1.} $[H_1, H_2, \ldots, H_{d + 1}] \notin \mathcal R^*_{\nu}(a)$.
Put
$$e_i(\nu; n_1, n_2, \ldots, n_{d + 1}) = \mathbf 0.$$

\noindent {\bf Case 2.} $[H_1, H_2, \ldots, H_{d + 1}] \in \mathcal R^*_{\nu}(a)$. Then, writing $B_i$ for the $i$th
simplicial cone corresponding to the generating $(d + 1)$-tuple $(H_1, H_2, \ldots, H_{d + 1})$, put
\begin{equation}\label{eq:ei_n}
e_i(\nu; n_1, n_2, \ldots, n_{d + 1}) = \bigl( a - \wgt_{\nu}(H_1, H_2, \ldots, H_{d + 1}) \bigr) e(B_i; \nu).
\end{equation}
Here $e(B_i; \nu)$ denotes, as in Section~\ref{sect:centralray}, the central vector of the cone $B_i$ with respect to the measure $\nu$.

The definition of $e_i(\nu; n_1, n_2, \ldots, n_{d + 1})$ is complete. Let us prove the following continuity property.

\begin{lem}\label{lem:continuous_n}
Let $\nu \in \mathcal M^{\circ}_a(V)$ be fixed. Then $e_i(\nu; n_1, n_2, \ldots, n_{d + 1})$ is continuous as a function from
$({\mathbb S}^{d - 1})^{d + 1}$ to $V$.
\end{lem}

\begin{proof}
Consider the two cases.

\noindent {\bf Case 1.} $e_i(\nu; n_1, n_2, \ldots, n_{d + 1}) = \mathbf 0$ holds. Then there are two subcases.

\noindent {\bf Subcase 1.1.} $[H_1, H_2, \ldots, H_{d + 1}] \notin \mathcal R^*_{\nu}(a)$. By Lemma~\ref{lem:compact}, the set
$\mathcal R^*_{\nu}(a)$ is compact. Therefore
$$[H(n'_1), H(n'_2), \ldots, H(n'_{d + 1})] \notin \mathcal R^*_{\nu}(a)$$
for any $(d + 1)$-tuple $[n'_1, n'_2, \ldots, n'_{d + 1}]$ of unit vectors close enough to $[n_1, n_2, \ldots, n_{d + 1}]$.
Hence
$$e_i(\nu; n'_1, n'_2, \ldots, n'_{d + 1}) \equiv \mathbf 0$$
in some neighborhood of $[n_1, n_2, \ldots, n_{d + 1}]$.

\noindent {\bf Subcase 1.2.} $[H_1, H_2, \ldots, H_{d + 1}] \in \mathcal R^*_{\nu}(a)$ and $\wgt_{\nu}(H_1, H_2, \ldots, H_{d + 1}) = a$.
Let arbitrary $\varepsilon > 0$ be given. Then for any $(d + 1)$-tuple $[n'_1, n'_2, \ldots, n'_{d + 1}]$ of unit vectors close enough
to $[n_1, n_2, \ldots, n_{d + 1}]$ one has
$$\wgt_{\nu}(H(n'_1), H(n'_2), \ldots, H(n'_{d + 1})) \geq a - \varepsilon.$$
Consequently, $\| e_i(\nu; n'_1, n'_2, \ldots, n'_{d + 1}) \| < \varepsilon$. This ends the proof of the subcase, since $\varepsilon$ is arbitrary.

\noindent {\bf Case 2.} $e_i(\nu; n_1, n_2, \ldots, n_{d + 1}) \neq \mathbf 0$. Then
$$[H_1, H_2, \ldots, H_{d + 1}] \notin \mathcal R^*_{\nu}(a) \quad \text{and} \quad \wgt_{\nu}(H_1, H_2, \ldots, H_{d + 1}) < a.$$
Let arbitrary $\varepsilon > 0$ be given. Then, if a $(d + 1)$-tuple $[n'_1, n'_2, \ldots, n'_{d + 1}]$ of unit vectors is close enough
to $[n_1, n_2, \ldots, n_{d + 1}]$, one has
$$\wgt_{\nu}(H(n'_1), H(n'_2), \ldots, H(n'_{d + 1})) < a.$$
The property (R2) and the uniqueness part of the property (R1) imply
$$[H(n'_1), H(n'_2), \ldots, H(n'_{d + 1})] \in \mathcal R^*_{\nu}(a).$$
Thus in some neighborhood of $[n_1, n_2, \ldots, n_{d + 1}]$ the function $e_i$ is defined according to~\eqref{eq:ei_n}. But both multipliers in
the right-hand side of~\eqref{eq:ei_n} are continuous (the second one due to Proposition~\ref{prop:central-vec-contunuous}), hence Case 2 follows.

\end{proof}

We continue by defining the function $e_i(\nu): \mathcal M^{\circ}_a(V) \to V$
as follows:
\begin{equation}\label{eq:ei}
e_i(\nu) = \int\limits_{(\mathbb S^{d - 1})^{d + 1}} e_i(\nu; n_1, n_2, \ldots, n_{d + 1}) \, dn_1 dn_2 \ldots dn_{d + 1}.
\end{equation}
Due to Lemma~\ref{lem:continuous_n}, the integration is indeed possible.

We emphasize that the ordered $(d + 1)$-tuple
$$[e_1(\nu), e_2(\nu), \ldots, e_{d + 1}(\nu)]$$
depends on the choice of $\mathcal R^*_{\nu}(a)$ from $(d + 1)!$ possible variants,
but the unordered $(d + 1)$-tuple
$$( e_1(\nu), e_2(\nu), \ldots, e_{d + 1}(\nu) )$$
does not. Therefore we have a map $T^V_a : \mathcal M^{\circ}_a(V) \to V^{d + 1} / \mathfrak S_{d + 1}$ defined by
$$T^V_a(\nu) = ( e_1(\nu), e_2(\nu), \ldots, e_{d + 1}(\nu) ).$$

Our aim will be to show that $T^V_a$ satisfies the requirements of Lemma~\ref{lem:structure}. We do it in the next lemmas,
leaving aside property 2, which is straightforward from the definition of $T^V_a$.

\begin{lem}\label{lem:interior2}
For every $\nu \in \mathcal M^{\circ}_a(V)$ one has $T^V_a(\nu) \in \mathcal T(V)$.
\end{lem}

\begin{proof}
Equivalently, we have to prove
\begin{equation}\label{eq:inc}
\mathbf 0 \in \relint \conv \{ e_1(\nu), e_2(\nu), \ldots, e_{d + 1}(\nu) \}.
\end{equation}

Assertion 1 of Lemma~\ref{lem:median} implies that
$$\wgt_{\nu}(H_1, H_2, \ldots, H_{d + 1}) = \depth_{\nu}(\mathbf 0) < a$$
for some generating $(d + 1)$-tuple $(H_1, H_2, \ldots, H_{d + 1})$. Consequently,
$$e_i(\nu; n_1, n_2, \ldots, n_{d + 1}) \not\equiv \mathbf 0.$$
Let $B_i$ denote the simplicial cones corresponding to the $(d + 1)$-tuple
$(H_1, H_2, \ldots, H_{d + 1})$.

Choose an arbitrary $(d + 1)$-tuple
$$[H'_1, H'_2, \ldots, H'_{d + 1}] \in \mathcal R^*_{\nu}(a)$$
satisfying
$$\wgt_{\nu}(H'_1, H'_2, \ldots, H'_{d + 1}) < a.$$
If $[B'_1, B'_2, \ldots, B'_{d + 1}]$ is the corresponding $(d + 1)$-tuple of simplicial cones, and $n'_j$
is the outer unit normal to $H'_j$, then
\begin{equation}\label{eq:inside}
e_i(\nu, n'_1, n'_2, \ldots, n'_{d + 1}) \in \relint C(B'_i) \subset \relint B_i.
\end{equation}
Note that \eqref{eq:inside} holds in a set of positive measure, for instance, in some neighborhood of
$[n_1, n_2, \ldots, n_{d + 1}]$.

If a $(d + 1)$-tuple of unit vectors $[n'_1, n'_2, \ldots, n'_{d + 1}]$ cannot be obtained in such a way
\begin{equation}\label{eq:zero}
e_i(\nu; n'_1, n'_2, \ldots, n'_{d + 1}) = \mathbf 0.
\end{equation}

Integrating according to~\eqref{eq:ei}, one obtains that
$$e_i(\nu) \in \relint B_i,$$
and, in particular, $e_i(\nu) \neq \mathbf 0$.

Applying Lemma~\ref{lem:interior}, we immediately get \eqref{eq:inc}.

\end{proof}

\begin{lem}\label{lem:aux-ei-continuous}
Let $\nu \in \mathcal M^{\circ}_a(V)$. Writing $a_1 = \tfrac{a_0 + a}{2}$, assume that there is an infinite sequence of measures $\nu_j \in \mathcal M^{\circ}_a(V)$
($j = 1, 2, \ldots$) such that $\| \nu_j - \nu \| < \tfrac{a_1 - a}{2}$ for each $j$, and $\lim\limits_{j \to \infty} \nu_j = \nu$.
Let $\mathcal R^*_{\nu_j}(a) \subseteq \mathcal R^*_{\nu}(a_1)$ for each $j$, and also $\mathcal R^*_{\nu}(a) \subseteq \mathcal R^*_{\nu}(a_1)$.
Then for every sequence $n_1, n_2, \ldots n_{d + 1}$ ($n_k \in \mathbb S^{d - 1}$) and every $i = 1, 2, \ldots, d + 1$ one has
$$\lim\limits_{j \to \infty} e_i(\nu_j; n_1, n_2, \ldots n_{d + 1}) = e_i(\nu; n_1, n_2, \ldots, n_{d + 1}).$$
\end{lem}

\begin{proof}
Consider the two possible cases.

\noindent {\bf Case 1.} $e_i(\nu; n_1, n_2, \ldots, n_{d + 1}) = \mathbf 0$. If, as before, $H_j$ is a half-space
with outer normal $n_j$ and $\mathbf 0 \in \partial H_j$, then there are three subcases.

\noindent {\bf Subcase 1.1.} $[H_1, H_2, \ldots, H_{d + 1}] \notin \mathcal R^*_{\nu}(a_1)$. Then for every $j$
$$[H_1, H_2, \ldots, H_{d + 1}] \notin \mathcal R^*_{\nu_j}(a_1),$$
hence $e_i(\nu_j; n_1, n_2, \ldots, n_{d + 1}) = \mathbf 0$.

\noindent {\bf Subcase 1.2.} $[H_1, H_2, \ldots, H_{d + 1}] \in \mathcal R^*_{\nu}(a_1)$ and $\wgt_{\nu}(H_1, H_2, \ldots, H_{d + 1}) \geq a$.
Let arbitrary $\varepsilon > 0$ be given. Then for every $j \geq j_0$ one has
$$\wgt_{\nu_j}(H_1, H_2, \ldots, H_{d + 1}) > a - \varepsilon,$$
hence $\| e_i(\nu_j; n_1, n_2, \ldots, n_{d + 1}) \| < \varepsilon$. Since $\varepsilon$ is arbitrary, Subcase 1.2 is proved.

\noindent {\bf Subcase 1.3.} $[H_1, H_2, \ldots, H_{d + 1}] \in \mathcal R^*_{\nu}(a_1)$, $\wgt_{\nu}(H_1, H_2, \ldots, H_{d + 1}) < a$, but
$$[H_1, H_2, \ldots, H_{d + 1}] \notin \mathcal R^*_{\nu}(a).$$
We claim that this subcase is impossible. Indeed, (R1) implies that there exists a non-trivial permutation $\sigma$ of $\{1, 2, \ldots, d + 1\}$ such that
$$[H_{\sigma(1)}, H_{\sigma(2)}, \ldots, H_{\sigma(d + 1)}] \in \mathcal R^*_{\nu}(a) \subseteq \mathcal R^*_{\nu}(a_1).$$
A contradiction to the uniqueness part of (R1) applied to $\mathcal R^*_{\nu}(a_1)$.

\noindent {\bf Case 2.} $e_j(\nu; n_1, n_2, \ldots, n_{d + 1}) \neq \mathbf 0$. Then
$$\wgt_{\nu}(H_1, H_2, \ldots, H_{d + 1}) < a.$$
Consequently, for some $j_0$ and every $j > j_0$ one has
$$\wgt_{\nu_j}(H_1, H_2, \ldots, H_{d + 1}) < a.$$
Then $[H_1, H_2, \ldots, H_{d + 1}] \in \mathcal R^*_{\nu_j}(a)$. Indeed, otherwise there exists a non-trivial permutation
$\sigma$ of $\{1, 2, \ldots, d + 1\}$ such that
$$[H_{\sigma(1)}, H_{\sigma(2)}, \ldots, H_{\sigma(d + 1)}] \in \mathcal R^*_{\nu_j}(a) \subseteq \mathcal R^*_{\nu}(a_1),$$
which leads to a contradiction similarly to Subcase 1.3.

Hence for $j > j_0$ the vector $e_i(\nu_j; n_1, n_2, \ldots, n_{d + 1})$ is defined according to~\eqref{eq:ei_n}. We have
$$\lim\limits_{j \to \infty} \wgt_{\nu_j}(H_1, H_2, \ldots, H_{d + 1}) = \wgt_{\nu}(H_1, H_2, \ldots, H_{d + 1}),$$
$$\text{and} \lim\limits_{j \to \infty} e(B_i; \nu_j) = e(B_i; \nu)$$
(the last identity is due to Proposition~\ref{prop:central-vec-contunuous}). Hence Case 2 follows.

\end{proof}

\begin{cor}\label{cor:t-continuous}
The map $T^V_a(\nu)$ is continuous.
\end{cor}

\begin{proof}
Choose an arbitrary measure $\nu_0 \in \mathcal M^{\circ}_a(V)$. Let us prove the continuity of $T_a(\nu)$
at $\nu = \nu_0$.

Since Cauchy and Heine definitions of continuity are equivalent in our case, we will use the latter. I.e.,
for an arbitrary sequence $\nu_j \to \nu_0$ ($j = 1, 2, \ldots$, $\nu_j \in \mathcal M^{\circ}_a(V)$) we will prove
$$\lim\limits_{j \to \infty} T_a(\nu_j) = T_a(\nu_0).$$

Without loss of generality we can assume that $\| \nu_j - \nu_0 \| < a_1 - a$, where $a_1 = \tfrac{a + a_0}{2}$.
Choose $\mathcal R^*_{\nu_0}(a_1)$ and the families
$$\mathcal R^*_{\nu_j}(a) \subseteq \mathcal R^*_{\nu_0}(a_1) \quad (j = 0, 1, 2, \ldots)$$
satisfying the requirements of Lemma~\ref{lem:ordering-in-neighborhood},

By Lemma~\ref{lem:aux-ei-continuous}, the sub-integral
function for $\nu_j$ in~\eqref{eq:ei} converges pointwise to that of $\nu_0$. Also, by definition,
$$\|e_i(\nu_j; n_1, n_2, \ldots, n_{d + 1}) \| \leq a.$$
Hence, by the Bounded Convergence Theorem (see, for example,~\cite[Section 4.2]{RoF2010}), $e_i(\nu_j)$
converges to $e_i(\nu)$.

\end{proof}

\section{Acknowledgements}

The authors acknowledge the hospitality of the Alfr\'ed R\'enyi Mathematical Institute, Budapest, where the research was done.
We are thankful to Imre B\'ar\'any for making the joint research possible, and for participating in inspiring discussions.
We appreciate the extremely useful comments and suggestions by Roman Karasev who changed our understanding of the topological part.
The first author also thanks Boris Bukh, Bo'az Klartag, Micha Sharir, and Gabriel Nivasch for useful discussions of the result.

\end{document}